%% file: ex_article.tex
\documentclass[hidelinks,onefignum,onetabnum]{siamart251216}

\input{ex_shared}

\newcommand{\LRs}[1]{\left[ #1 \right]}  
\newcommand{\LRp}[1]{\left( #1 \right)}  
\newcommand{\LRc}[1]{\left\{ #1 \right\}} 
\newcommand{\nor}[1]{\left\| #1 \right\|}
\newcommand{\snor}[1]{\left| #1 \right|}

\newcommand{\mc}[1]{\mathcal{#1}}
\renewcommand{\epsilon}{\varepsilon}

\renewcommand{\L}{\mathcal{L}}
\usepackage{xcolor}
\usepackage{rotating}
\usepackage{tcolorbox}
\newenvironment{question}{
\begin{tcolorbox}[colframe=blue!75!black, colback=blue!10!white]
\begin{center}
}
{
\end{center}
\end{tcolorbox}
}

\newenvironment{idea}{
\begin{tcolorbox}[colframe=green!60!black, colback=green!10!white, title={\bf Idea}]
\begin{center}
}
{
\end{center}
\end{tcolorbox}
}

\ifpdf
\hypersetup{
  pdftitle={From an Elementary Proof of Error Representation for Hermite Quadrature to a Rediscovery of Legendre Polynomials and Rodrigues Formula},
  pdfauthor={Tan Bui-Thanh, Giancarlo Villatoro, and C. G. Krishnanunni}
}
\fi

\begin{document}

\maketitle

\begin{abstract}
We generalize  two-point interpolatory Hermite quadrature to functions with available values and the first $n-1$ derivatives at both end points.
Armed with integration by parts in the reverse form we provide an elementary derivation of an exact error represenation of Hermite quadrature rule. This approach possesses several advantages over the classical approaches: i) Only integration by parts is needed for the derivation; ii) the error representation requires much milder regularity,  namely the existence of $n$th-order derivative rather than a $(2n)$th-order derivative of the function under consideration. As a result, our error formula is valid for less regular functions for which the classical ones are not valid; iii) our approach rediscovers Legendre polynomials and more interestingly it provides a surprisingly elegant relation between Legendre polynomial and Hermite interpolation. In particular, Legendre polynomials are precisely the error kernels for interpolatory Hermite quadrature rules; and iv) We also rediscover the Rodrigues formula for Legendre polynomials as part of our findings.  For those who are interested in a different proof of the exact error representation for Hermite quadrature rule, we provide an alternative proof using the Peano kernel theorem. We also provide a composite interpolatory Hermite quadrature rule for practical applications.
\end{abstract}



\section{Introduction}
Numerical integration is a cornerstone of scientific computing, and the Newton--Cotes formulas are among the most widely used techniques. 
These methods approximate the integral of a function using only its values at selected points in the interval. 
The general  Newton--Cotes formula for approximating the integral of $f(x)$ on $[a,\ b]$ takes the form \cite{quarteroni2006numerical}:
\begin{equation}
\int_a^b f(x)\, dx= \sum_{i=0}^n w_i f(x_i)+\mathrm{error}, 
\label{newton_cotes}
\end{equation}
where $w_i$ are the weights, $\{x_i\}_{i=0}^n$ are $(n+1)$ equally spaced points in the interval $[a,\ b]$ with $x_0=a$ and $x_{n}=b$, and it is assumed that the values $f(x_i)$ are known at these points. 
For instance, the case $n=1$ is referred to as the \emph{trapezoidal rule}, and $n=2$ corresponds to \emph{Simpson's rule}. 
A conventional derivation of the error estimates for the trapezoidal and Simpson's rules is based on Newton's divided differences formula and the integral mean value theorem \cite{atkinson}. 
A much simpler technique, relying on reverse integration by parts, was developed by Cruz et al.\ \cite{cruz2003elementary} to derive error bounds for the trapezoidal rule. 
Subsequent work extended this approach to obtain analogous estimates for Simpson's rule \cite{ELSINGERJASONR.2007AEPO, HaiD.D.2008AEPo}. 
This method has attracted our attention due to its simplicity, requiring only an elementary background in calculus to derive an exact error representation. 
Moreover, it  requires less regularity, and thus valid for broader classes of functions, in contrast to the conventional (classical) approach.

In this work, we consider the setting in which, in addition to the function values $f(x_i)$, one also has access to derivative information at these points.  
The natural extension for quadrature in this case is to first use Hermite interpolation \cite{BirkhoffDeBoor1964,Schumaker1973,Lourakis2007,GasperRahman2004,Davis1975,StoerBulirsch1993} for $f\LRp{x}$, and then integrate. 
In particular, we define and investigate the following Hermite quadrature rule:
\begin{equation}
  \int_a^b f(x)\, dx=\int_a^b H_n(f;\ x)\, dx+E_n, 
  \label{eq:hermiteQuadrature}
\end{equation}
where $H_n(f;\ x)$ is the two-point Hermite interpolating polynomial constructed using the following information on the function and its $n-1$ derivative values at the two end points: 
$f(a),\ f^{(1)}(a),\dots, f^{(n-1)}(a)$ and $f(b),\ f^{(1)}(b),\dots ,f^{(n-1)}(b)$, and $E_n$ is the corresponding integration error by replacing $f(x)$ with its Hermite interpolation $H_n\LRp{f;\ x}$. This quadrature rule belongs to a larger class of interpolatory quadature rules \cite{Scott2011NumericalAnalysis}.

To motivate why \cref{eq:hermiteQuadrature} can provide a better approximation to the integral than the trapezoidal rule in \cref{newton_cotes}, let us consider an example. 
Consider the function $f(x)=x^2\sin{(x)}$ on the interval $[0,\ \pi]$, as shown in \cref{area}. 
Given $f(0)$, $f(\pi)$, $f^{(1)}(0)$ and $f^{(1)}(\pi)$, our aim is to approximate the area under the curve (shaded region in \cref{area}) in two different ways: 
(i) using the Newton--Cotes formula \cref{newton_cotes}; and 
(ii) using the Hermite quadrature \cref{eq:hermiteQuadrature}. 
We have the following estimates for the integral: 
\begin{equation}
\begin{aligned}
    & \int_0^\pi f(x)\, dx=\pi^2-4,\quad \mathrm{(true\ value)},\\
    &\int_0^\pi f(x)\, dx\approx \frac{\pi-0}{2}\LRp{f(0)+f(\pi)}=0,\quad \mathrm{(trapezoidal\ rule)},\\
    &\int_0^\pi f(x)\, dx\approx \int_0^\pi H_2(f;\ x)\, dx=\int_0^\pi -(x^3-\pi x^2)\, dx=\frac{\pi^4}{12},\quad \mathrm{(Hermite)},
    \end{aligned}
    \label{estimates}
\end{equation} 
where $H_n(f;\ x)$ is the two-point Hermite interpolating polynomial \cite{alma991022147289706011}. 
Note that from \cref{estimates} the error between the trapezoidal rule estimate and the true value is $\pi^2-4\approx 5.869$, whereas the error between the Hermite quadrature estimate and the true value is approximately $2.247$, which is  noticeably more accurate.
This improvement is due to its use of both function and derivative values, which allows it to capture more information about the function's behavior over the interval.
Thus, in scenarios where derivative information is available or easy to compute, Hermite-based quadrature rules present a compelling alternative to classical Newton--Cotes methods. 

\begin{figure}[h!]      
          \centering
          \includegraphics[scale=0.5]{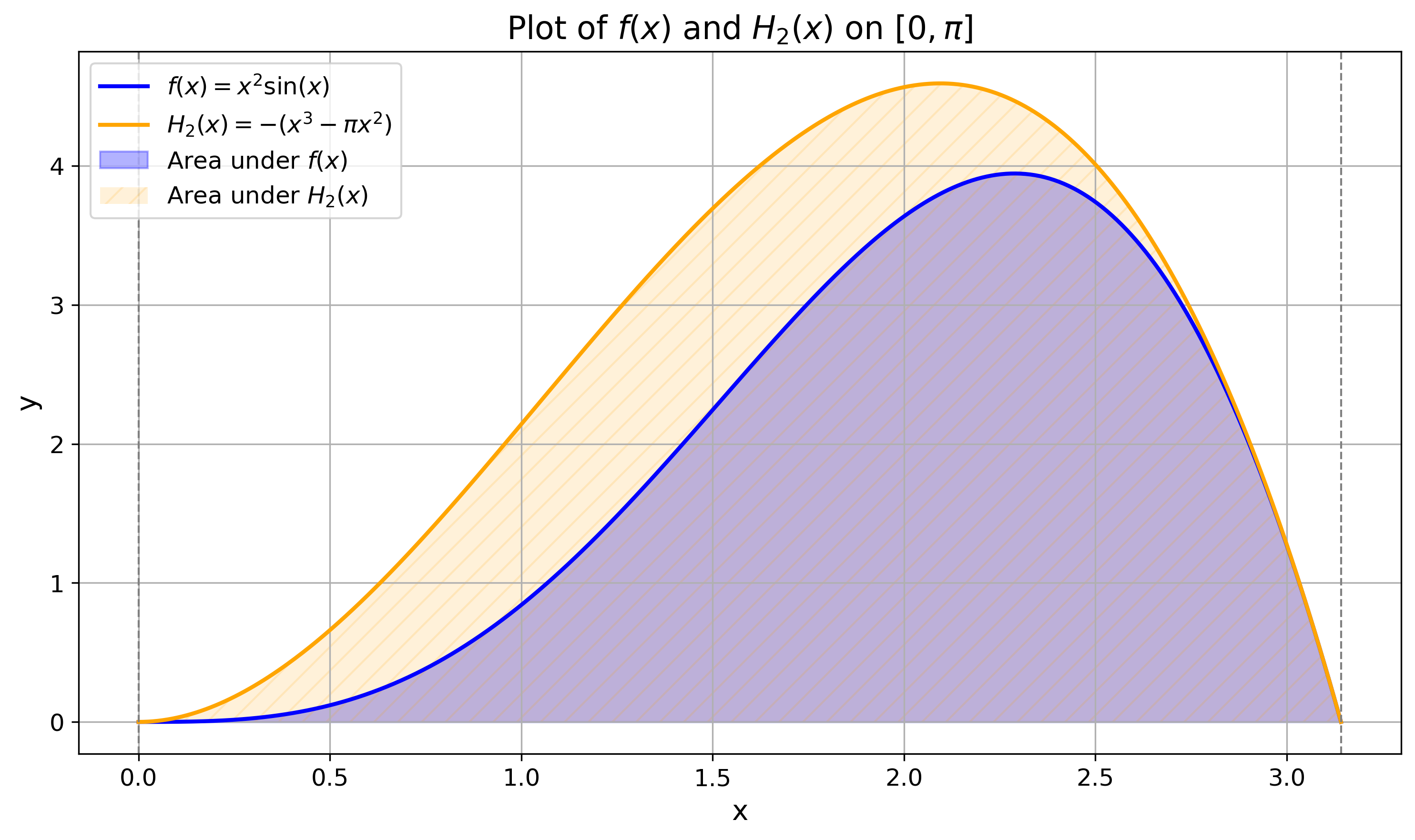}
    \caption{Plot of the function $f(x)=x^2\sin{(x)}$ with the shaded region denoting the area under the curve between $a=0$ and $b=\pi$. Note that $\int_a^b f(x)\, dx$ equals the area under the curve. The curve corresponding to the trapezoidal rule is identically $0$ and therefore does not appear in the plot.}
  \label{area}
\end{figure}

In the example considered above, the true value of the integral was available, which allowed us to compute the error for different quadrature rules. 
In practice, the true value of the integral is unavailable, and we are interested in deriving both exact error representations and their upper bounds.

 Classical expression, requiring the availability of $f^{(2n)}$, for the error in \cref{eq:hermiteQuadrature} is given in \cite{alma991022147289706011, atkinson}. 
Lampret et al.\ \cite{LampretVito2004AItH} considered a quadrature rule referred to as the ``composite Hermite rule'' for an interval $\LRs{a,b}$ with $k$ subintervals. It applies when both function values $f(a),\ f(a+\frac{b-a}{k}),\ f(a+2\frac{b-a}{k}),\dots ,f(b)$ and the endpoint derivatives $f^{(1)}(a),\ f^{(1)}(b)$ are available.
They also provided an expression for the error. 
Note that the case $k=1$ in Lampret et al.\ \cite{LampretVito2004AItH} corresponds to taking $n=2$ in \cref{eq:hermiteQuadrature}, and the resulting error bounds agree with those in \cite{alma991022147289706011, atkinson}. Two-point interpolatory Hermite quadrature rule with $n=2$ have been also considered and derived using the Peano kernel theorem \cite{Scott2011NumericalAnalysis}.
However,  these approaches require the availability of the $2n$-th derivative of the integrand. 

To the authors' best knowledge, the use of elementary techniques such as reverse integration by parts for deriving the error of Hermite quadrature rule \cref{eq:hermiteQuadrature} has not been previously investigated. 
While one reverse integration by parts  is sufficient to derive the error for the trapezoidal rule \cite{cruz2003elementary}, several additional technical developments and extensions are required to derive the error expression for the Hermite quadrature rule. First, by invoking the Beta function \cite{abromowitz1972handbook} and the General Leibniz Rule \cite{apostol1967one}, we derive the explicit formula for the Hermite quadrature rules. Second, 
we need to perform reverse integration by parts (RIBP) $n$ times in order to yield the desirable form for Hermite quadrature rule and its error. A key finding is that our error expression requires the existence of $n$-order derivative of the integrand instead of $2n$-order derivative. Our result is thus valid over larger classes of functions as oppose to the classical one.  Third, for the residual of the $n$ times RIPB to be the exact expression for Hermite quadrature rules, free parameters from RIPB must be identified to match Hermite quadrature rules.  Through several techinical theorems, we show that there is a unique set of parameters and provide a recursive formula for the parameters. Fourth is a pleasant surprise. In particular, by accident we rediscover Legendre polynomials through the Hermite quadrature error kernel. Another surprising finding is a rediscovery of the Rodrigues formula for the Legendre polynomials. Fifth, we provide an alternative derivation and proof for the error of the Hermite quadrature rule, when $2n$-order derivative is available, using the Peano kernel theorem. For practical purposes, a composite interpolatory Hermite quadrature rule was also provided.

\section{Hermite interpolating function and its integral}

Let $f(x)$ be a function defined on the interval $[a,b]$, where $a < b$. 
Suppose $f(a)$, $ f^{(1)}(a),\dots, f^{(n-1)}(a)$ and $f(b)$, $ f^{(1)}(b),\dots, f^{(n-1)}(b)$ are available. We can approximate $f(x)$  using the two-point Hermite interpolation polynomial.  
The two-point Hermite interpolation polynomial utilizing the aforementioned information is defined as follows
\cite{alma991022147289706011, atkinson}, assuming all the quantities exist:

\begin{definition}[Two-Point Hermite Interpolating Polynomial]
\label{defi:her_def}
    Let $a$ and $b$ be distinct. The Hermite interpolating polynomial is defined by:
    \begin{equation}
    \begin{aligned}
         &H_n(f;\ x) = (x-a)^{n}\sum_{k=0}^{n-1}\frac{B_k(x-b)^k}{k!} + (x-b)^{n}\sum_{k=0}^{n-1}\frac{A_k(x-a)^k}{k!},\\
        &\mathrm{with}\quad A_k = \frac{d^k}{dx^k}\LRs{\frac{f(x)}{(x-b)^{n}}}_{x=a}, \quad B_k = \frac{d^k}{dx^k}\LRs{\frac{f(x)}{(x-a)^{n}}}_{x=b}.
        \end{aligned}
        \label{her_def}
    \end{equation}
\end{definition}

It is easy to see that  $H_n(f;\ x)$ satisfies 
     \begin{align*}
         H_n(f;\ a) = f(a),\ H_n^{(1)}(f;\ a) = f^{(1)}(a), \dots, H_n^{(n-1)}(f;\ a) = f^{(n-1)}(a), \\
         H_n(f;\ b) = f(b),\ H_n^{(1)}(f;\ b) = f^{(1)}(b), \dots, H_n^{(n-1)}(f;\ b) = f^{(n-1)}(b),
     \end{align*}
and the classical error in interpolation can be shown to have this the following form \cite{alma991022147289706011, atkinson}:
     \begin{equation}
         f(x) - H_n(f;\ x) = \frac{f^{(2n)}(\varepsilon)}{(2n)!}(x-a)^n(x-b)^n,
    \label{her_error}
     \end{equation}
     where $\varepsilon := \varepsilon\LRp{x} \in [a,b]$. Consequently, the error for Hermite quadrature, as an example of interpolatory quadrature class \cite{Scott2011NumericalAnalysis}, using the classical interpolation error \cref{her_error} can be written as
     \begin{equation}
\int_a^b f(x)\ dx=\int_a^b H_n(f;\ x)\ dx+\int_a^b\frac{f^{(2n)}(\epsilon)}{(2n)!}(x-a)^n(x-b)^n\ dx. 
\label{int_h}
\end{equation}
For the classical Hermite interpolation error to be meaningful,  $(2n)$-th derivative of $f(x)$ needs to exist in some sense (e.g. a sufficient condition\footnote{$L_2\LRp{a,b}$ and $C\LRs{a,b}$ denote the space of square integrable functions on $\LRp{a,b}$ and continuous functions $\LRs{a,b}$, respectively.} is $f^{(2n)}\LRp{x} \in L_2\LRp{a,b}$ or  $f^{(2n)}\LRp{x} \in C\LRs{a,b}$).  
For certain functions such as 
\[ 
f(x) := \int_0^x \log\LRp{t-1/2}\,dt
\]
on the interval $[0,\ 1]$, the classical error \cref{her_error} does not exist for $n=1$, though the Hermite interpolation (simple linear interpolation for this case), and thus its error, is well-defined.

\begin{question}
    Can we  derive and bound the  Hermite quadrature error $E_n$ in \cref{eq:hermiteQuadrature} using only up to  $n$th-order derivative of $f\LRp{x}$?
\end{question}
\vspace{0.1 cm}

{{In the following, we provide an answer to the aforementioned question  by using elementary techniques.
To that end, let us first write $\int_a^b H_n(f;\ x)\, dx$ in \cref{int_h} in the form
\begin{equation}
\int_a^b H_n(f;\ x)\, dx=\sum_{j=0}^{n-1}w_j^a f^{(j)}(a)+\sum_{j=0}^{n-1}w_j^b f^{(j)}(b),    
\label{her_rewrite}
\end{equation}
where $f^{(j)}(a)$ denotes the $j$-th derivative of $f(x)$ evaluated at $x=a$, and $w_j^a,\ w_j^b$ are the weights to be determined.
\section{Hermite quadrature from Hermite interpolation}

\begin{proposition}
\label{first_p}
    Consider the Hermite interpolating polynomial in \cref{her_def} for a given $n$. 
    Then the weights $w_j^a$ and $w_j^b$ in \cref{her_rewrite} are given by:
\[ w_j^a=(b-a)^{j+1}n\sum_{k=j}^{n-1}\binom{k}{j}\frac{(n+k-j-1)!}{(n+k+1)!}, \quad
 w_j^b=(-1)^j w_j^a,
\]
where $\binom{k}{j}=\frac{k!}{j!(k-j)!}$. Consequently,
\[
\int_a^b H_n(f;\ x)\, dx=\sum_{j=0}^{n-1}w_j^a \LRs{f^{(j)}(a)+ (-1)^j f^{(j)}(b)}.
\]
\end{proposition}
\begin{proof}
We begin with the left hand side of \cref{her_rewrite} using \Cref{defi:her_def}
\begin{equation}
\begin{aligned}
    \int_a^b H_n(f;\ x)\, dx
    &= \underbrace{\int_a^b (x-a)^{n}\sum_{k=0}^{n-1}\frac{B_k(x-b)^k}{k!}\, dx}_{(I)} \\
    &\quad+ \underbrace{\int_a^b (x-b)^{n}\sum_{k=0}^{n-1}\frac{A_k(x-a)^k}{k!}\, dx}_{(II)}.
\end{aligned}
\label{I_II}
\end{equation}
Let us consider the term (I) in \cref{I_II}. We have:
\begin{align}
\label{1}
        & \int_a^b(x-a)^{n}\sum_{k=0}^{n-1}\frac{B_k(x-b)^k}{k!}\, dx = \sum_{k=0}^{n-1}\frac{B_k}{k!}\int_a^b(x-a)^{n}(x-b)^k\, dx 
    \\
    \label{2}
    &= \sum_{k=0}^{n-1}\frac{B_k}{k!}\int_0^1t^{n}(b-a)^{n}(b-a)^k(t-1)^k(b-a)\, dt, \\
    \label{3}
    &= \sum_{k=0}^{n-1}\frac{B_k(b-a)^{n+k+1}}{k!}\int_0^1t^{n}(t-1)^k\,dt
    \\
    \label{4}
    &= \sum_{k=0}^{n-1}\frac{B_k(b-a)^{n+k+1}}{k!}(-1)^kB(n+1,k+1)
    \\
      \label{5}
    &=\sum_{k=0}^{n-1}\frac{B_k(b-a)^{n+k+1}}{k!}(-1)^k\frac{n!k!}{(n+k+1)!},
\end{align}
where we used the transformation $x=t(b-a)+a$ in going from \cref{1} to \cref{2}, the definition of the Beta function \cite{abromowitz1972handbook}, $B(n+1,\ k+1)$, in \cref{4}, and its value in \cref{5}. 
A similar computation for the term (II) in \cref{I_II} yields: 
\begin{equation}
        \int_a^b(x-b)^{n}\sum_{k=0}^{n-1}\frac{A_k(x-a)^k}{k!}\, dx = \sum_{k=0}^{n-1}\frac{A_k(b-a)^{n+k+1}}{k!}(-1)^{n}\frac{n!k!}{(n+k+1)!}.
        \label{II_comp}
\end{equation}
We now simplify the terms $A_k$ and $B_k$ from \cref{her_def}.  
Using the General Leibniz Rule \cite{apostol1967one}, we have:
\begin{equation}
\begin{aligned}
 B_k &= \frac{d^k}{dx^k}\LRs{\frac{f(x)}{(x-a)^n}}_{x=b}
    \\
    &= \LRs{\sum_{j=0}^k\binom{k}{j}f^{(j)}(x) \frac{d^{k-j}}{dx^{k-j}}\LRs{\frac{1}{(x-a)^n}}}_{x=b}
    \\
    &= \sum_{j=0}^kf^{(j)}(b)\LRs{\binom{k}{j}\frac{(-1)^{k-j}(n+k-j-1)!}{(n-1)!(b-a)^{n+k-j}}}.    
\end{aligned}
\label{B_k}
\end{equation}
Following a similar procedure for $A_k$ in \cref{her_def} yields:
\begin{equation}
 A_k = \sum_{j=0}^kf^{(j)}(a)\LRs{\binom{k}{j}\frac{(-1)^{k-j}(n+k-j-1)!}{(n-1)!(a-b)^{n+k-j}}}.
    \label{A_k}
\end{equation}
Substituting \cref{B_k} and \cref{A_k} back into \cref{5} and \cref{II_comp}, respectively, and recomputing \cref{I_II} leads to: 
\begin{equation}
\begin{aligned}
    \int_a^bH_n(f;\ x)\, dx &= \sum_{j=0}^{n-1}f^{(j)}(b)\LRs{(-1)^j(b-a)^{j+1}n\sum_{k=j}^{n-1}\binom{k}{j}\frac{(n+k-j-1)!}{(n+k+1)!}}
    \\
    &\quad+ \sum_{j=0}^{n-1}f^{(j)}(a)\LRs{(b-a)^{j+1}n\sum_{k=j}^{n-1}\binom{k}{j}\frac{(n+k-j-1)!}{(n+k+1)!}}.
\end{aligned}       
\label{her_pro}
\end{equation}
This concludes the proof.
\end{proof}
\section{An elementary approach to derive Hermite quadrature error  using milder conditions}

We start by recalling the reverse integration by parts formula:
\begin{equation*}
    \int_a^b f(x)\, dx = f(b)(b+c)+ f(a)(-a-c) - \int_a^b f^{(1)}(x)(x+c)\, dx,
\end{equation*}
where $c \in \mathbb{R}$ is arbitrary. 
Performing reverse integration by parts process once more gives
\begin{multline}
    \int_a^b f(x)\, dx = f(b)(b+c) + f(a)(-a-c)+f^{(1)}(b)\LRp{-\frac{1}{2}(b+c)^2 - \delta_0} \\
    \quad+f^{(1)}(a)\LRp{\frac{1}{2}(a+c)^2 + \delta_0} +\int_a^b f^{(2)}(x)\LRp{\frac{1}{2}(x+c)^2+\delta_0}\, dx,
\label{match}
\end{multline}
 where  $\delta_0 \in \mathbb{R}$ is arbitrary.  
More generally, if we repeat the procedure $n$ times, we  arrive at the following $n$th-time reverse integration by parts formula:
\begin{proposition}[$n$th-time reverse integration by parts]
   \\ Assume $f(x)\in C^{(n)}[a,\ b]$,\footnote{$C^{(n)}[a,\ b]$ denotes the set of all functions that are $n$ times continuously differentiable on the closed interval $[a,\ b]$.} where $n \in \mathbb{N}$. The following $\LRp{n-1}$th-time reverse integration by parts hold:
\begin{equation}
    \boxed{
\begin{aligned}
\int_a^b f(x)\, dx 
&= \sum_{j=0}^{n-1} 
f^{(j)}(b)(-1)^j 
\left[
    \frac{(b+c)^{j+1}}{(j+1)!}
    + \sum_{i=0}^{j-1} \delta_{\,i+n-1-j}\frac{b^{i}}{i!}
\right] \\
&\quad + \sum_{j=0}^{n-1} 
f^{(j)}(a)(-1)^{j+1}
\left[
    \frac{(a+c)^{j+1}}{(j+1)!}
    + \sum_{i=0}^{j-1} \delta_{\,i+n-1-j}\frac{a^{i}}{i!}
\right] \\
&\quad + \int_a^b (-1)^n f^{(n)}(x)
\left(
    \frac{(x+c)^n}{n!}
    + \sum_{i=0}^{n-2} \delta_i \frac{x^i}{i!}
\right)\,dx,
\end{aligned}
}
\label{eq:iterate}
\end{equation}
where the set of free parameters  $\boldsymbol{\theta} := \LRc{c,\ \delta_0,\dots, \delta_{n-2}}\in \mathbb{R}^n$ are arbitrary.
\label{propo:RIBP}
\end{proposition}
\begin{proof}
    Let us define the kernel $K_n$ as
\begin{equation}
K_n(x,\boldsymbol{\theta}) := \frac{(x+c)^n}{n!} + \sum_{i=0}^{n-2} \delta_i \frac{x^i}{i!},
\label{eq:Kn}
\end{equation}
and it follows that $K_n^{(n)}\LRp{x,\boldsymbol{\theta}} = 1$. The $n$-fold integration by parts identity gives
\begin{multline*}
    \int_a^b f^{(n)}(x) K_n(x,\boldsymbol{\theta})\, dx
= \sum_{k=0}^{n-1} (-1)^k 
\left[ f^{(n-1-k)}(x) K_n^{(k)}(x,\boldsymbol{\theta}) \right]_{x=a}^{x=b}
 \\+(-1)^n \int_a^b f(x) K_n^{(n)}(x,\boldsymbol{\theta})\,dx.
\end{multline*}
Since $K_n^{(n)}(x,\boldsymbol{\theta})=1$, multiplying by $(-1)^n$ and renaming the indices for the derivatives we obtain
\[
\begin{aligned}
\int_a^b f(x)\,dx
&= \sum_{j=0}^{n-1} f^{(j)}(b)(-1)^j K_n^{(n-1-j)}(b,\boldsymbol{\theta})
   \;+\; \sum_{j=0}^{n-1} f^{(j)}(a)(-1)^{j+1} K_n^{(n-1-j)}(a,\boldsymbol{\theta}) \\
&\quad\;\; + \int_a^b (-1)^n f^{(n)}(x) K_n(x,\boldsymbol{\theta})\, dx.
\end{aligned}
\]
Straightforward differentiation gives
\[
K_n^{(n-1-j)}(x,\boldsymbol{\theta})
= \frac{(x+c)^{j+1}}{(j+1)!}
  + \sum_{i=0}^{j-1} \delta_{i+n-1-j}\, \frac{x^i}{i!},
\qquad 0 \le j \le n-1,
\]
and this concludes the proof.
\end{proof}

\begin{remark}
    Note we assume $f \in C^{(n)}\LRs{a,b}$ in \cref{propo:RIBP} for simplicity. The result still holds under milder conditions. For example, $f \in C^{(n-1)}\LRp{a,b}$ so as $f^{(2n)} \in L^2\LRp{a,b}$ is a milder sufficient condition.
\end{remark}

\begin{idea}
If we can identity the set of  constants $\boldsymbol{\theta} := \LRc{c,\ \delta_0,\dots, \delta_{n-2}}$ in \cref{eq:iterate} such that the corresponding coefficients of $f^{(j)}(a)$ and $f^{(j)}(b)$ in \cref{eq:iterate} and \cref{her_pro} match, then an explicit formula for error $E_n$ in the Hermite quadrature \cref{eq:hermiteQuadrature} has been found using milder condition.
\end{idea}
Before we carry out the proposed idea for a general $n$, which is more complex,  let us first provide the details for $n=2$ in which we can explicitly derive  all quantities of interest,  including the exact expression for $E_2$ and its bounds.

\section{$n=2$: Exact expression for $E_2$ and its bounds}
In this section, thanks to our exact expression for $E_2$, for the first time in the literature we can derive the explicit expression for $E_2$ in terms of $f^{(2)}\LRp{x}$ by unique judicious choices for $c$ and $\delta_0$. We then provide new bounds for $E_2$ in terms of $f^{(2)}$, $f^{(3)}$, and $f^{(4)}$. Furthermore, we recover the exact classical expression for $E_2$ in terms of $f^{(4)}$.

\subsection{Exact expression and characterizations of $E_2$}

\begin{theorem}[Exact expression for $E_2$]
\label{thm:the_one}
   Consider the case $n=2$ in \Cref{her_def}, where we assume that $f(a),\ f(b),\ f^{(1)}(a)$, and $f^{(1)}(b)$ are available. 
   Assume that $f(x)\in C^{(2)}[a,\ b]$. 
   Then
   \[
   E_2 = \int_a^b f^{(2)}(x)K_2(x)\, dx = \frac{1}{2}\int_a^b f^{(2)}(x)(x-a)(x-b)\, dx + \frac{(b-a)^2}{12}[f^{(1)}(b) - f^{(1)}(a)],
   \]
   where the kernel is:
\begin{equation}
K_2(x) = \frac{(x-a)(x-b)}{2} + \frac{(b-a)^2}{12}
\label{eq:kernel}
\end{equation}
\end{theorem}

\begin{proof}
For $n=2$, \cref{her_pro} can be simplified as:
\begin{equation}
\begin{aligned}
    \int_a^b H_2(x)\, dx
    &= f(b)\LRp{\frac{b-a}{2}} + f(a)\LRp{\frac{b-a}{2}} \\
    &\quad+ f^{(1)}(b)\LRp{-\frac{(b-a)^2}{12}} + f^{(1)}(a)\LRp{\frac{(b-a)^2}{12}}.
\end{aligned}
\label{int_h_2}
\end{equation}
Matching the coefficients between the reverse integration by parts formula \cref{match} and the integral of the Hermite interpolation \cref{int_h_2} yields the following system of equations involving the unknowns $c$ and $\delta_0$:
\[b+c = \frac{b-a}{2},\quad  -a-c = \frac{b-a}{2},\]
\[-\frac{1}{2}(b+c)^2-\delta_0 = -\frac{(b-a)^2}{12},\quad   \frac{1}{2}(a+c)^2+\delta_0 = \frac{(b-a)^2}{12}.\]
It is easy to see that $c = -\frac{a+b}{2}$ and $\delta_0 = - \frac{(b-a)^2}{24}$ satisfy this system of equations. 
Therefore, from \cref{match} the error in the Hermite interpolation-based quadrature is
\[ \int_a^b f(x)\,dx - \int_a^b H_2(x)\,dx
=\int_a^b f^{(2)}(x)\LRp{\frac{1}{2}\LRp{x-\frac{a+b}{2}}^2- \frac{(b-a)^2}{24}}\,dx,
\]
and this ends the proof.
\end{proof}

Compared to the classical error representation \cref{int_h}, the result in \cref{thm:the_one} requires only $f^{(2)}$ instead of $f^{(4)}$, and thus is valid for a broader class of functions.

\begin{lemma}[Kernel Properties]
\label{lem:propertiesK}
Define 
\[
G(x) := \int_a^x K_2(t)\, dt,
\]
then $K_2(x)$ and $G\LRp{x}$ satisfy:
\begin{enumerate}
\item $\int_a^b K_2(x)\, dx = 0$, $\int_a^b G\LRp{x}\,dx = 0$, 
\item $\int_a^b |K(x)|\, dx = \frac{(b-a)^3\sqrt{3}}{54}$, and $\nor{K_2}_2 = \sqrt{\int_a^b |K_2(x)|^2\, dx} = \frac{(b-a)^{5/2}}{\sqrt{720}}$
\item $\int_a^b |G(x)|\, dx = \frac{(b-a)^4}{192}$, and $\nor{G}_2 = \sqrt{\int_a^b |G(x)|^2\, dx} = \frac{(b-a)^{7/2}}{\sqrt{30240}}$.
\end{enumerate}
\end{lemma}
\begin{proof}
    The proof is straight forward by direct algebraic calculations.
\end{proof}

\subsection{General bounds for $E_2$ with uniform and $L_2$ norms}

\begin{proposition}
\label{propo:L2Inf_norm}
Define the midrange of the second derivative as 
\[
\beta_2:= \frac{\inf_{\LRs{a,b}} f^{(2)}(x) 
 +\sup_{\LRs{a,b}} f^{(2)}(x) }{2}
 \]
and the mean  as $\overline{f^{(2)}} = \frac{1}{b-a}\int_a^b f^{(2)}(x)\, dx$. Furthermore, define the deviations as
\[
\delta(x) := {f^{(2)}(x) - \beta_2}, \quad \text{and} \quad
\Delta(x) := {f^{(2)}(x) - \overline{f^{(2)}}}.
\]
Then
\begin{equation*}
|E_2| \leq \nor{\delta}_\infty \cdot \frac{(b-a)^3\sqrt{3}}{54}, \text{ and }
|E_2| \leq \nor{\Delta}_2 \frac{(b-a)^{5/2}}{\sqrt{720}}.
\end{equation*}
\end{proposition}
\begin{proof}
    Since $\int_a^b K_2(x)\, dx = 0$:
\begin{align*}
\snor{E_2} &= \snor{\int_a^b f^{(2)}(x)K_2(x)\, dx} = \snor{\int_a^b \delta(x)K_2(x)\, dx} \le \nor{\delta}_\infty\int_a^b \snor{K_2(x)}\, dx \\
& \le \nor{\delta}_\infty \cdot \frac{(b-a)^3\sqrt{3}}{54}.
\end{align*}
On the other hand, by Cauchy-Schwarz inequality we have
\[
\snor{E_2} = \snor{\int_a^b \Delta(x)K_2(x)\, dx} \le
\nor{\Delta}_2\nor{K_2}_2 = \nor{\Delta}_2 \frac{(b-a)^{5/2}}{\sqrt{720}}.
\]
\end{proof}

\begin{remark}
    If we do not exploit the fact that $\int_a^b K(x)\, dx = 0$, then the obvious bounds would be
    \[
    |E_2| \leq \nor{f^{(2)}}_\infty \cdot \frac{(b-a)^3\sqrt{3}}{54}, \text{ and }
|E_2| \leq \nor{f^{(2)}}_2 \frac{(b-a)^{5/2}}{\sqrt{720}},
    \]
    which would be more conservative owing the fundamental results for the deviation from the midrange in the infinity norm and the mean in the $L^2$-norm:
    \[
    \nor{\delta}_\infty \le \nor{f^{(2)}}_\infty, \text{ and } \nor{\Delta}_2 \le \nor{f^{(2)}}_2.
    \]
Note the constants $\beta_2$ and $\overline{f^{(2)}}$ are the best in the sense that they are chosen to solve the following optimization problems
    \[
\beta_2 = \arg\min_{\beta} \nor{f^{(2)} - \beta}_\infty, \quad \text{ and } \quad
\overline{f^{(2)}} = \arg\min_{\beta} \nor{f^{(2)} - \beta}_2.
    \]
\end{remark}

\subsection{Improved Bounds for $E_2$ when $f^{(3)}$ exists}
Using integration by parts, we have
\begin{align*}
E_2 &= \int_a^b f^{(2)}(x)K_2(x)\, dx \\
&= \LRs{f^{(2)}(x) \int_a^x K_2(t)\, dt}_a^b - \int_a^b f^{(3)}(x) G(x)\, dx = -\int_a^b f^{(3)}(x) G(x)\, dx,
\end{align*}
where we have used properties of $G(x)$ in \Cref{lem:propertiesK}.
\begin{proposition}
    \label{propo:L2Inf_normf3}
Define the midrange of the third derivative as \\
$\beta_3:= \frac{\inf_{\LRs{a,b}} f^{(3)}(x) 
 +\sup_{\LRs{a,b}} f^{(3)}(x) }{2}$
and
the mean as $\overline{f^{(3)}} = \frac{1}{b-a}\int_a^b f^{(3)}(x)\, dx$. In addition, define the deviations as
\[
\lambda(x) := {f^{(3)}(x) - \beta_3}, \quad
\text{and }
\Lambda(x) := {f^{(3)}(x) - \overline{f^{(3)}}}.
\]
Then
\begin{equation*}
|E_2| \leq \nor{\lambda}_\infty \cdot  \frac{(b-a)^4}{192}, \text{ and }
|E_2| \leq \nor{\Lambda}_2 \frac{(b-a)^{7/2}}{\sqrt{30240}}.
\end{equation*}
\end{proposition}
\begin{proof}
    The proof is similar to that of \Cref{propo:L2Inf_norm} using properties of $G\LRp{x}$ in \Cref{lem:propertiesK}, and thus is omitted.
\end{proof}

\subsection{An enhance of the classical representation for $E_2$ when $f^{(4)}$ exists}

\begin{proposition}
    \label{propo:L2Inf_normf4}
    There holds:
    \[
    E_2 = \int_a^b f^{(4)}(x) H(x)\, dx,
\quad    
    \text{where} \quad
H(x) :=  \int_{a}^x G(t)\,dt =  \frac{(x-a)^2(x-b)^2}{24}.
\]
    If, additionally, $f^{(4)}(x)$ is continuous on $\LRp{a,b}$, then
    \[
    E_2 = f^{(4)}(\varepsilon)\int_a^b  H(x)\, dx  = f^{(4)}(\varepsilon) \frac{\LRp{b-a}^5}{720},
    \]
    for some $\varepsilon \in \LRs{a,b}$.

\end{proposition}
\begin{proof}
Integrating by parts  gives
\[
E_2=-\int_a^b f^{(3)}(x) G(x)\, dx =
 -\LRs{f^{(3)}(x) H(t)\, dt}_a^b + 
\int_a^b f^{(4)}(x) H(x)\, dx.
\]
Since $H(x)$ does not change sign (nonnegative) on $\LRs{a,b}$, an application of the integral mean value theorem yields the second assertion.
\end{proof}

\begin{remark}
    The result in \cref{propo:L2Inf_norm} is slightly more favorable than the classical one in \cref{int_h} as it does not requires an additional constant $\epsilon\LRp{x}$.  
\end{remark}
\section{General $n$ case}

\subsection{Exact error representation via coefficient matching}

We have shown in the previous section how to determine free parameters $c,\delta_0$ for the special case $n = 2$. We can deploy a similar approach to find free parameters for $n\in \LRc{3,4,5}$, and the results are summarized in \Cref{tab:deltas}.
In the proof of \cref{thm:the_one}, we notice that there are 4 conditions, but we have only 2 parameters $c$ and $\delta_0$. It is, however, not an overdetermined system as two of them are redundant as we now show for general $n$. To begin, let us rewrite \cref{eq:iterate} in the following form
\begin{equation}
\label{eq:RIBP-general}
\begin{aligned}
\int_a^b f(x)\,dx
&= \sum_{j=0}^{n-1} \alpha_j^b f^{(j)}(b)
   + \sum_{j=0}^{n-1} \alpha_j^a f^{(j)}(a)
   + \int_a^b (-1)^n f^{(n)}(x)\,K_n(x,\boldsymbol{\theta})\,dx,
\end{aligned}
\end{equation}
where the kernel $K_n$ is defined in \cref{eq:Kn}, and 
\begin{align*}
\alpha_j^b
&:= (-1)^j\left[
    \frac{(b+c)^{j+1}}{(j+1)!}
    + \sum_{k=0}^{j-1} \delta_{\,n-1-j+k}\,\frac{b^k}{k!}
\right],\\
\alpha_j^a
&:= (-1)^{j+1}\left[
    \frac{(a+c)^{j+1}}{(j+1)!}
    + \sum_{k=0}^{j-1} \delta_{\,n-1-j+k}\,\frac{a^k}{k!}
\right].
\end{align*}

\begin{idea}
\Cref{eq:RIBP-general} is the error equation for the  Hermite quadrature in \cref{her_rewrite} when
\begin{multline}
\sum_{j=0}^{n-1} \alpha_j^b f^{(j)}(b)
   + \sum_{j=0}^{n-1} \alpha_j^a f^{(j)}(a) \\ = 
   \sum_{j=0}^{n-1}w_j^a f^{(j)}(a)+\sum_{j=0}^{n-1}w_j^b f^{(j)}(b), \quad \forall f \in C^{(n)}[a,\ b].
   \label{eq:HermiteCondition}
\end{multline}
\end{idea}

\begin{lemma}[Matching conditions]
    The condition in \cref{eq:HermiteCondition} holds iff
    \begin{equation}
            \alpha_j^a = w_j^a,  \text{ and }
    \alpha_j^b = w_j^b, \quad \forall j=0, \hdots, n-1.
\label{eq:HermiteMatching}
    \end{equation}
    \label{lem:HermiteMatching}
\end{lemma}
\begin{proof}
    The sufficiency is clear, and we focus on the necessary. We start by defining the difference linear functional
\[
D(f):=\sum_{j=0}^{n-1}(\alpha_j^b-w_j^b) f^{(j)}(b)
   + \sum_{j=0}^{n-1}(\alpha_j^a-w_j^a) f^{(j)}(a).
\]
By \eqref{eq:HermiteCondition}, we have $D(f)=0$ for all $f\in C^{(n)}[a,b]$.
Let us define $d_j^a:=\alpha_j^a-w_j^a$ and $d_j^b:=\alpha_j^b-w_j^b$. There exists a smooth cut-off function $\phi$ such that 
$\phi =  1$ on a neighborhood of $a$,
and $\phi\equiv 0$ on a neighborhood of $b$ \cite{folland1999real, hirsch1976differential}.
Next, let us define
\begin{equation}
f_{j_0}(x):=\frac{(x-a)^{j_0}}{j_0!}\,\phi(x).
\label{eq:fj}
\end{equation}
Because $\phi= 1$ near $a$, all derivatives of $\phi$ vanish at $a$, so
\[
f_{j_0}^{(j)}(a)=\delta_{j,j_0}\qquad (j=0,\dots,n-1).
\]
where $\delta_{j,j_0}$ is the  Kronecker delta. Because $\phi= 0$ near $b$, $f_{j_0} =  0$ near $b$, hence
\[
f_{j_0}^{(j)}(b)=0\qquad (j=0,\dots,n-1).
\]
Substituting into $D(f)=0$ gives
\[
0=D(f_{j_0})
=\sum_{j=0}^{n-1} d_j^a f_{j_0}^{(j)}(a)
+\sum_{j=0}^{n-1} d_j^b f_{j_0}^{(j)}(b)
=d_{j_0}^a,
\]
so $d_{j_0}^a=0$. Since $j_0$ was arbitrary, $d_j^a=0$ for all $j$.
Similarly, we can show that  $d_j^b=0$ for all $j$.
Therefore $\alpha_j^a=w_j^a$ and $\alpha_j^b=w_j^b$ for all $j=0,\dots,n-1$.
\end{proof}

\subsection{Redundancy in matching conditions and properties of the error kernel}

In the matching conditions \cref{eq:HermiteMatching}, there are  $2n$ equations for $n$ unknowns  $\boldsymbol{\theta} = \LRc{c,\delta_{n-2},\hdots,\delta_0}$. We first consider $j=0$. In this case, \cref{eq:HermiteMatching} becomes
\[
\begin{aligned}
    b + c &= (b-a)n\sum_{k=0}^{n-1}\binom{k}{j}\frac{(n+k-j-1)!}{(n+k+1)!} = \frac{b-a}{2},\\
    -(a+c) &= (b-a)n\sum_{k=0}^{n-1}\binom{k}{j}\frac{(n+k-j-1)!}{(n+k+1)!} = \frac{b-a}{2},
\end{aligned}
\]
owing to the fact that 
\[
n \sum_{k=0}^{n-1}\frac{(n+k-1)!}{(n+k+1)!} = \frac{1}{2}.
\]
Though there are 2 equations for $c$ they are not independent. Solving either of them for $c$ gives
\[
c = -\frac{a+b}{2},
\]
independent of $n$. That is, if $\alpha_0^a = w_0^a$, then $\alpha_0^b = w_0^b$, and vice versa. We are going to show that this also holds for all $1\le j \le n-2$. To that end, assume $\alpha_0^a = w_0^a$, we need somehow to tie the RIBP \cref{eq:iterate} and the Hermite quadrature \cref{her_rewrite} in order to relate $\alpha_j^b  $ and $ w_j^b$. 

The classical result \cref{int_h} gives us the desired connection by taking $f \in \mc{P}^{2n-1}\LRs{a,b}$, where $\mc{P}^{2n-1}\LRs{a,b}$ is the set of polynomials of order at most $2n-1$, as the Hermite interpolation error $\frac{f^{(2n)}(\varepsilon)}{(2n)!}\LRp{(x-a)(x-b)}^n$ vanishes. In particular, we have
\begin{multline*}
\sum_{j=0}^{n-1}w_j^a f^{(j)}(a)+\sum_{j=0}^{n-1}w_j^b f^{(j)}(b) = \int_a^b H_n(f;\ x) dx = \int_a^b f(x) dx \\= \sum_{j=0}^{n-1}\alpha_j^a f^{(j)}(a)+\sum_{j=0}^{n-1}\alpha_j^b f^{(j)}(b) + \int_a^b (-1)^n f^{(n)}(x)
K_n\LRp{x,\boldsymbol{\theta}}\,dx,
\end{multline*}
for any $f \in \mc{P}^{2n-1}\LRs{a,b}$.
After using the assumption $\alpha_j^a = w_j^a$, we obtain
\begin{equation}
\sum_{j=0}^{n-1}\LRp{w_j^b - \alpha_j^b} f^{(j)}(b) = (-1)^n \int_a^b  f^{(n)}(x) 
K_n\LRp{x,\boldsymbol{\theta}}\,dx,
\label{eq:orthogonality}
\end{equation}
for any $f \in \mc{P}^{2n-1}\LRs{a,b}$. What remains is to choose $f \in \mc{P}^{2n-1}\LRs{a,b}$ judiciously so that the right hand side vanishes while isolating $\LRp{w_j^b - \alpha_j^b}$ out on the left hand side. To annihilate the right hand side, the simplest choice is $f \in \mc{P}^{n-1}\LRs{a,b}$, and to single out $\LRp{w_j^b - \alpha_j^b}$ we exploit what we already knew in \cref{eq:fj}. By taking $f = \frac{(x-a)^{j_0}}{j_0!}$ for $0 \le j_0 \le n-1$, we have
\[
f^{(j)}(a)=\delta_{j,j_0}, \quad \text{ and thus } \quad w_{j_0}^b - \alpha_{j_0}^b = 0, \quad \forall 0 \le j_0 \le n-1,
\]
which concludes the proof of the following redundancy theorem.
\begin{theorem}[Redundancy]
     $\alpha_j^a = w_j^a, 0\le j \le n-1$ if and only if $\alpha_j^b = w_j^b, 0\le j \le n-1$.
     \label{thm:redundancy}
\end{theorem}

\subsection{Polynomial kernels $K_n\LRp{x,\boldsymbol{\theta}^*}$ are shifted Legendre polynomials}
Now, revisiting \cref{eq:orthogonality} and using \cref{thm:redundancy} yield interesting properties of the polynomial kernel $K_n\LRp{x,\boldsymbol{\theta}}$.
\begin{corollary}
    Suppose there exists  a set of parameters $\boldsymbol{\theta}^*$ such that $\alpha_j^a = w_j^a, 0\le j \le n-1$. Then
    \begin{equation}
        \int_a^b  f(x) K_n\LRp{x,\boldsymbol{\theta}^*}\,dx = 0, \quad \forall f \in \mc{P}^{n-1}\LRs{a,b},
\label{eq:KernelOrthogonality}
    \end{equation}
    which in turns implies that $K_n\LRp{x,\boldsymbol{\theta}^*}$ is, up to a sign, symmetric around $x = (a+b)/2$, i.e.,
\begin{equation}
    K_n\LRp{a+b-x,\boldsymbol{\theta}^*}  = 
(-1)^nK_n\LRp{x,\boldsymbol{\theta}^*}.
\label{eq:KnSymmetry}
\end{equation}
\label{coro:orthogonality}
\end{corollary}
\begin{proof}
    We prove the second assertion, as the  first one is clear. Let us define
    \[
    P\LRp{x} : = K_n\LRp{a+b-x,\boldsymbol{\theta}^*}  -
(-1)^nK_n\LRp{x,\boldsymbol{\theta}^*}.
    \]
Then, $P\LRp{x}$ is a polynomial of order at most $n-1$. Now from \cref{eq:orthogonality} we have, for any $f \in \mc{P}^{n-1}\LRs{a,b}$,
\[
\int_a^b  f(a+b-x) K_n\LRp{x,\boldsymbol{\theta}^*}\,dx = 0,
\]
which implies
\[
\int_a^b  f(x) K_n\LRp{a+b-x,\boldsymbol{\theta}^*}\,dx = 0.
\]
Combining the last two equations gives
\[
\int_a^b  f(x) P\LRp{x}\,dx = 0, \quad \forall f \in \mc{P}^{n-1}\LRs{a,b},
\]
which, after taking $f(x) = P\LRp{x}$, becomes
\[
\int_a^b  P^2\LRp{x}\,dx = 0,
\]
and this concludes the proof.
\end{proof}
\begin{remark}[$K_n$ is the shifted unnormalized $n$th-order Legendre polynomial]
    The orthogonality condition \cref{eq:orthogonality} of the kernel $K_n$ shows that it is  the (shifted and unnormalized) $n$th-order Legendre polynomial. In other words, we have rediscovered Legendre polynomials.
\end{remark}

\subsection{Rediscovery of Rodrigues formula when $f^{(2n)}$ is available}
In this section, we explore the orthogonality conditions of the polynomial kernels $K_n$ in \cref{eq:orthogonality}, and the classical error representation \cref{int_h} of the Hermite quadrature to derive the Rodrigues formula for Legendre polynomials. We begin by definining $j$th antiderivative for $K_n$ as follows
\[
K_n^j\LRp{x,\boldsymbol{\theta}^*} = \int_a^x K_n^{j-1}\LRp{y,\boldsymbol{\theta}^*}\,dy, \quad j = 1, \cdots,n,
\]
where
\[
K_n^0\LRp{x,\boldsymbol{\theta}^*} := K_n\LRp{x,\boldsymbol{\theta}^*}.
\]

In the following, we suppress $\boldsymbol{\theta}^*$ for readability.

\begin{lemma}
    There holds:
    \[
    K_n^j\LRp{a} = K_n^j\LRp{b} = 0, \quad \forall j=1,\cdots,n,
    \]
    and thus there exist a constant $C$ such that
    \begin{equation}
            K_n^n\LRp{x} = C\LRp{x-a}^n\LRp{x-b}^n.
    \label{eq:Knn}
    \end{equation}
    \label{lem:Knn}
\end{lemma}
\begin{proof}
    We proceed by induction and by definition we only need to show \\ $K_n^j\LRp{b,\boldsymbol{\theta}^*} = 0$. For $j = 1$, it is clear from the orthogonality condition \cref{eq:orthogonality} (by taking $f = 1$) that 
    \[
     K_n^1\LRp{b} = \int_{a}^b K_n\LRp{y}\,dy = 0.
    \]
    Now suppose that the assertion holds up to $j$ with $j \le n-1$, that is, $K_n^\ell\LRp{b} = 0$ for $1 \le \ell \le j\le n-1$.
    By taking $f = x^{j}/{j!}$ in \cref{eq:orthogonality} 
    and by integrating by parts $j$ times, we have
    \begin{multline*}
    0 = \int_{a}^b \frac{x^{j}}{j!} K_n\LRp{x}\,dx = \sum_{k=0}^{j-1} (-1)^k 
\left[ K_n^{k+1}(x) \frac{x^{j-k}}{\LRp{j-k}!} \right]_{x=a}^{x=b} \\  + \LRp{-1}^{j} \int_{a}^b  K_n^{j}\LRp{x}\,dx = \LRp{-1}^{j} \int_{a}^b  K_n^{j}\LRp{x}\,dx = \LRp{-1}^{j} K_n^{j+1}\LRp{b},
    \end{multline*}
where we have used the induction hypothesis in the third equality.

For the second assertion, we note that $K_n^n$ is a polynomial of order at most $2n$ since $K_n^0$ is a polynomial of order $n$. On the other hand, the first assertion implies that $K_n^n$ and its first $n-1$ derivatives vanish at both $a$ and $b$. Thus $K_n^n$ must have the desirable form in \cref{eq:Knn}.
\end{proof}

\begin{theorem}[Rediscovery of the Rodrigues formula]
    There holds:
    \[
    K_n^n\LRp{x} = \frac{1}{\LRp{2n}!}(x-a)^n(x-b)^n,
    \]
    or equivalently
    \[
    K_n\LRp{x} = \frac{1}{\LRp{2n}!}\frac{d^{n} }{dx^{n}}(x-a)^n(x-b)^n.
    \]
    \label{thm:Rodrigues}
\end{theorem}
\begin{proof}
 Combining the classical error representation \cref{int_h} and our error repressentation \cref{eq:RIBP-general} yields
    \begin{multline*}
        \int_a^b\frac{f^{(2n)}(\varepsilon)}{(2n)!}\LRp{(x-a)(x-b)}^n\, dx = \int_a^b (-1)^n f^{(n)}(x)\,K_n(x)\,dx, \\
        = \sum_{k=0}^{n-1} (-1)^{k+n} 
\left[ f^{(n+k)}(x) K_n^{k+1}(x) \right]_{x=a}^{x=b}
 + \int_a^b f^{(2n)}(x) K_n^{n}(x)\,dx \\
 = \int_a^b f^{(2n)}(x) K_n^{n}(x)\,dx, \quad \forall f \in C^{(2n)}\LRs{a,b},
    \end{multline*}
    where we have applied the integration by parts $n$ times in the second equality and \cref{lem:Knn} in the last equality. Finally taking $f\LRp{x} = x^{2n}$ yields $C = 1/(2n)!$ and this completes the proof.
\end{proof}

\begin{remark}
    Note that \cref{thm:Rodrigues} is the Rodrigues formula for unnormalized and shifted Legendre functions $K_n\LRp{x}$.
\end{remark}

\subsection{Uniqueness of the parameters from matching conditions}

We are now in a position to study the existence and uniqueness of the free parameters from \cref{eq:HermiteCondition}. The answer turns out to be straighforward. From \cref{thm:redundancy}, we only need to consider the set of $n$ equations $\alpha_j^a = w_j^a, 0\le j \le n-1$ to solve for $n$ parameters $\boldsymbol{\theta}$.
These equations possess an obvious triangular structure that allows us to solve for all free parameters in a straightforward manner. In particular, for $j=0$, the corresponding equation $\alpha_0^a = w_0^a$ always yields $c = -\LRp{a+b}/2$ as we have shown above. For $j=1$, the corresponding equation $\alpha_1^a = w_1^a$ gives us a linear monomial equation in terms of $\delta_{n-2}$. If we continue this triangular structure, for general $1\le j \le n-1$, the corresponding equation $\alpha_j^a = w_j^a$ gives us a recursive formula for $\delta_{n-1-j}$:
\[
\delta_{n-1-j} = (-1)^{j+1}w_j^a -\frac{(a+c)^{j+1}}{(j+1)!}
 - \sum_{i=1}^{j-1} \delta_{\,i+n-1-j}\frac{a^{i}}{i!}.
\]

Let us denote $\boldsymbol{\theta}^*$ as the unique set of parameters from \cref{eq:HermiteCondition}. \Cref{tab:deltas} presents the closed-form expressions for all components of $\boldsymbol{\theta}^*$ for $n\in \LRc{2,3,4,5}$, as a function of $a$ and $b$. For general $n > 5$, we provide four components $c,\delta_{n-2},\delta_{n-3},\delta_{n-4}$ of $\boldsymbol{\theta}^*$.

\begin{sidewaystable}[h!]
\centering
\renewcommand{\arraystretch}{1.6}
\begin{tabular}{|c|c|c|c|c|}
\hline
$n$ 
& $c$ 
& $\delta_{n-2}$ 
& $\delta_{n-3}$ 
& $\delta_{n-4}$ 
\\
\hline
$2$ 
& $\displaystyle -\frac{a+b}{2}$ 
& $\displaystyle -\frac{(b-a)^2}{24}$ 
& -- 
& --
\\
\hline
$3$ 
& $\displaystyle -\frac{a+b}{2}$ 
& $\displaystyle -\frac{(b-a)^2}{40}$ 
& $\displaystyle \frac{(a+b)(b-a)^2}{80}$ 
& --
\\
\hline
$4$ 
& $\displaystyle -\frac{a+b}{2}$ 
& $\displaystyle -\frac{(b-a)^2}{56}$ 
& $\displaystyle \frac{(a+b)(b-a)^2}{112}$ 
& $\displaystyle 
  \frac{(b-a)^2\bigl((b-a)^2 -10(a+b)^2\bigr)}{4480}$
\\
\hline
$5$ 
& $\displaystyle -\frac{a+b}{2}$ 
& $\displaystyle -\frac{(b-a)^2}{72}$ 
& $\displaystyle \frac{(a+b)(b-a)^2}{144}$ 
& $\displaystyle 
  \frac{(b-a)^2\bigl((b-a)^2 -14(a+b)^2\bigr)}{8064}$
\\
\hline
$\text{General }n$ 
& $\displaystyle -\frac{a+b}{2}$ 
& $\displaystyle -\frac{(b-a)^2}{8(2n-1)}$ 
& $\displaystyle \frac{(a+b)(b-a)^2}{16(2n-1)}$ 
& $\displaystyle
  \frac{(b-a)^2\bigl((b-a)^2 + (6-4n)(a+b)^2\bigr)}
       {128(2n-3)(2n-1)}$
\\
\hline
\end{tabular}

\caption{Summary of kernel parameters 
$\boldsymbol{\theta}^* = \LRc{c,\delta_{n-2},\delta_{n-3},\delta_{n-4}}$ 
for the matching conditions \cref{eq:HermiteMatching}. The first four rows are for special cases $n\in \LRc{2,3,4,5}$, and the last row is for the first four free parameters for general $n$.}
\label{tab:deltas}
\end{sidewaystable}
\section{Error bounds for a general $n$ in \cref{eq:hermiteQuadrature}}
\label{sect:generaln}
\subsection{General bounds for $E_n$ with uniform and $L_2$ norms when only $f^{(n)}$ is available}
Recall that when the matching conditions in \cref{lem:HermiteMatching} hold, the exact error representation
\begin{equation}
E_n = \int_a^b (-1)^n f^{(n)}(x)\,K_n(x,\boldsymbol{\theta}^*)\,dx,
\label{eq:En}
\end{equation}
in \cref{eq:RIBP-general} requires only the $n$th-order derivative $f^{(n)}(x)$ to exist (such that $E_n$ is finite), whereas the classical error representation for Hermite quadrature in \cref{int_h} would demand $f^{(2n)}(x)$. Following similar steps that we carried out for $n=2$, we can obtain upper bounds for $E_n$. 

If only $f^{(n)}(x)$ is available (suppose that it is continuous), then the orthogonality result in \cref{eq:orthogonality} allows us to generalize the results in \cref{propo:L2Inf_norm}. Specifically, let us define the midrange of the $n$th-order derivative as
$\beta_n:= \frac{\inf_{\LRs{a,b}} f^{(n)}(x) 
 +\sup_{\LRs{a,b}} f^{(n)}(x) }{2}$
and
the mean as $\overline{f^{(n)}} = \frac{1}{b-a}\int_a^b f^{(n)}(x)\, dx$. In addition, define the deviations as
\[
\phi(x) := {f^{(n)}(x) - \beta_n}, \quad
\text{and }
\Phi(x) := {f^{(n)}(x) - \overline{f^{(n)}}}.
\]
\begin{proposition}
    There hold:
    \begin{equation*}
|E_n| \leq \nor{\phi}_\infty \int_{a}^b\snor{K_n\LRp{x,\boldsymbol{\theta}^*}}\,dx, \text{ and }
|E_n| \leq \nor{\Phi}_2 \nor{K_n\LRp{x,\boldsymbol{\theta}^*}}_2.
\end{equation*}
\end{proposition}

\subsection{General bounds for $E_n$ with uniform and $L_2$ norms when $f^{(n+k)}$ is available}

For this section we suppose that $f^{(n+k)}$, $0\le k < n-1$, are available and continuous. We again capitalize on the orthogonality relation \cref{eq:KernelOrthogonality} to upper-bound \cref{eq:En} using triangle inequality as
\begin{multline*}
    \snor{E_n} = \snor{\int_a^b  \LRs{f^{(n)}(x)-p_k\LRp{x}}\,K_n(x,\boldsymbol{\theta}^*)\,dx,} \\ \le 
\nor{f^{(n)}(x)-p_k\LRp{x}}_\infty
\int_a^b  \snor{K_n(x,\boldsymbol{\theta}^*)}\,dx,
\end{multline*}
where $p_k\LRp{x}$ is any polynomial of order at most $k$. The best polynomial $p_k^*\LRp{x}$ solves the following minimax problem
\[
p_k^*\LRp{x} = \arg\min_{p_k} \nor{f^{(n)}(x)-p_k\LRp{x}}_\infty,
\]
which is challenging to solve analytically. We can resort to a computational suboptimal solution $p_k^*\LRp{x}$ using Chebyshev interpolation at $k+1$ Chebyshev's nodes. If we take $a = -1$ and $b=1$ for clarity, the error bound is given as
\[
\snor{E_n} \le \frac{1}{2^{k}\LRp{k+1}!}\nor{f^{(n+k+1)}}_\infty \int_{-1}^{1} \snor{K_n(x,\boldsymbol{\theta}^*)}\,dx.
\]
On the other hand, if we invoke the Cauchy-Schwarz inequality, then the bound reads
\[
\snor{E_n} \le 
\nor{f^{(n)}-p_k}_2
\nor{K_n}_2. 
\]
The best polynomial $p_k^*\LRp{x}$ in this case solves the following least-squares problem
\[
p_k^*\LRp{x} = \arg\min_{p_k} \nor{f^{(n)}(x)-p_k\LRp{x}}_2,
\]
which can be solved in the closed forms, but we omit the details.

\section{Error derivation for Hermite quadrature using Peano kernel theorem}
In this section, we briefly discuss an alternative error derivation using Peano kernel theorem when the integrand resides in $C^{(2n)}\LRs{a,b}$. Let us define the following linear operator
\[
\L f := \int_{a}^b \LRs{f\LRp{x}-H_n\LRp{f;\ x}}\,dx,
\]
then $\L f = 0$ for all $f \in \mc{P}^{2n-1}\LRs{a,b}$ due to the exactness of Hermite quadrature. Furthermore, owing to the linearity of the integrals and the Hermite interpolations we have: $\forall g\LRp{x,y}\in \mc{P}^{2n}\LRs{a,b}$,
\begin{multline*}
\L \int_{a}^b g\LRp{x,y}\,dy = \int_{a}^b \int_{a}^b g\LRp{x,y}\,dy dx - \int_{a}^b H_n\LRp{\int_{a}^b g\LRp{x,y}\,dy;\ x} \,dx \\
= \int_{a}^b\int_a^b \LRs{g\LRp{x,y}\,dx - H_n\LRp{g\LRp{x,y};\ x}}\,dx\,dy = \int_a^b\L g\LRp{x,y}\,dy.
\end{multline*}
An application of the Peano kernel theorem \cite{Scott2011NumericalAnalysis,Iserles2008FirstCourseNADE} then gives
\[
\L f =  \int_a^b f^{(2n)}\LRp{x}K\LRp{x}\,dx, \quad \forall f \in C^{(2n)}\LRs{a,b},
\]
where the Peano kernel $K\LRp{x}$ is given by
\[
K\LRp{x} := \L \LRs{ \frac{\LRp{y-x}^{2n-1}_+}{\LRp{2n-1}!}}, \quad \text{ with }
\quad \LRp{y-x}_+ := \begin{cases}
    y-x & \text{if } y > x, \\
    0 & \text{otherwise.}
\end{cases}
\]
\begin{proposition}
    The Peano kernel theorem yields the same result as RIPB when $f^{(2n)}$ is available. In particular, $K\LRp{x} = \frac{1}{\LRp{2n}!} \LRp{x-a}^n\LRp{x-b}^n$.
\end{proposition}
\begin{proof}
    We sketch a proof in a few steps here.
    \begin{enumerate}
        \item Using the defintion of the linear operator $\L$ and the Peano kernel, we have
        \begin{equation}
                    K\LRp{x} = \frac{\LRp{b-x}^{2n}}{\LRp{2n}!} - \sum_{k=0}^{n-1}w_k^b\frac{\LRp{b-x}^{2n-1-k}}{\LRp{2n-1-k}!}.
        \label{eq:PeanoKernel}
        \end{equation}
        \item Since the lowest power for $\LRp{b-x}$ in $K\LRp{x}$ is $n$, $\LRp{b-x}^n$ divides $K\LRp{x}$.
        \item Using the exactness of Hermite quadrature for polynomial of order less than or equal $2n-1$, we can show that
        \[
        K^{(j)}\LRp{a} = 0, \quad \forall j=0,\hdots,n-1,
        \]
        which then implies that $\LRp{x-a}^n$ divides $K\LRp{x}$.
        \item At this point, we know that $K\LRp{x}$ must have the form
        \[
        K\LRp{x} = C\LRp{x-a}^n \LRp{x-b}^n,
        \]
        since the highest power for $x$ in \cref{eq:PeanoKernel} is $2n$. Since the coefficient for the highest power for $x$ in \cref{eq:PeanoKernel} is $\LRp{2n}!$, we conclude that
        \[
        K\LRp{x} = \frac{1}{\LRp{2n}!}\LRp{x-a}^n \LRp{x-b}^n,
        \]
        and this concludes the proof.
    \end{enumerate}
\end{proof}
\begin{remark}
    Similar to \cref{her_error}, the Peano kernel theorem approach also requires the same $C^{(2n)}$-regularity, which is much stringent than our approach that requires only $C^{(n)}$-regularity.
\end{remark}

\section{Composite Hermite quadrature rule}
This section briefly discuss an extension of the Hermite quadrature rule on one interval to multiple ones. Let $[a, b]$ be partitioned into $m$ subintervals:
\[
a = x_0 < x_1 < x_2 < \cdots < x_{m-1} < x_m = b,
\]
and define the subinterval widths as
\[
h_i = x_{i+1} - x_i, \quad i = 0, 1, \ldots, m-1.
\]
On each subinterval $[x_i, x_{i+1}]$, we use the two-point Hermite interpolation polynomial $H_n^{i}(f;\ x)$ that matches: i) $f^{(k)}(x_i)$ for $k = 0, 1, \ldots, n-1$, and ii) $f^{(k)}(x_{i+1})$ for $k = 0, 1, \ldots, n-1$. We define a composite Hermite quadrature as
\begin{multline*}
    \int_a^b f(x)\, dx \approx \sum_{i=0}^{m-1} \int_{x_i}^{x_{i+1}} H_n^{i}(f;\ x)\, dx \\= 
\sum_{i=0}^{m-1} \sum_{j=0}^{n-1} h_i^{j+1}\omega_j \LRs{(-1)^jf^{(j)}(x_{i+1}) + f^{(j)}(x_i)},
\end{multline*}
where
\[
\omega_j = n\sum_{k=j}^{n-1}\binom{k}{j}\frac{(n+k-j-1)!}{(n+k+1)!}.
\]

The error for the composite Hermite quadrature is simply a summation of the errors on all intervals, each of which is of the form we discussed above. We omit the details.
\section{Conclusion}

In this paper, we generalize  two-point interpolatory Hermite quadrature to functions with available values and the first $n-1$ derivatives at both end points. Using elementary integration by parts, we derive an exact error expression for the Hermite quadrature rule. We rigorously study the approach and its results. As a by-product of our study, we discovered that (shifted and unnormalized) Legendre polynomials are the error kernels of the Hermite quadrature rule, thus providing an elegant relationship between the classical Hermite interpolation and Legendre polynomails. Part of our finding is a rediscovery of the Rodrigues formula for Legendre polynomials. Besides these advantages over the classical derivation and error representation, our approach is valid over much broader class of integrands in $C^{(n)}\LRp{a,b}$ instead of $C^{(2n)}\LRs{a,b}$. For completeness, we also provide an alternative derivation using the Peano kernel theorem, and a composite Hermite quadrature rule.

\bibliographystyle{siamplain}
\bibliography{references}
\end{document}

%% file: ex_shared.tex
\usepackage{lipsum}
\usepackage{amsfonts}
\usepackage{graphicx}
\usepackage{epstopdf}

\usepackage{algorithmic}
\ifpdf
  \DeclareGraphicsExtensions{.eps,.pdf,.png,.jpg}
\else
  \DeclareGraphicsExtensions{.eps}
\fi

\newsiamremark{remark}{Remark}
\newsiamremark{hypothesis}{Hypothesis}
\crefname{hypothesis}{Hypothesis}{Hypotheses}
\newsiamthm{claim}{Claim}
\newsiamremark{fact}{Fact}
\crefname{fact}{Fact}{Facts}

\headers{On the Error of Hermite Quadrature: An
Elementary Proof}{Tan Bui-Thanh, Giancarlo Villatoro, and C G Krishnanunni}

\title{From an Elementary Proof of Error Representation for Hermite Quadrature to a Rediscovery of Legendre Polynomials and Rodrigues Formula\thanks{\href{https://users.oden.utexas.edu/~tanbui/books/HermiteQuadrature.pdf}{\textcolor{red}{Part of this work}} was carried out during Moncrief Summer Internship Program 2025 at the Oden Institute for Computational Engineering and Sciences.
}}

\author{Tan Bui-Thanh\thanks{Oden Institute for Computational Engineering $\&$ Sciences; Department of Aerospace Engineering $\&$ Engineering Mechanics, University of Texas at Austin
  (\email{tanbui@oden.utexas.edu}).}
\and Giancarlo Villatoro\thanks{Department of Mathematics, University of Texas at Austin, 
  (\email{ggv258@my.utexas.edu}.}
\and C G Krishnanunni\thanks{Department of Aerospace Engineering $\&$ Engineering Mechanics, University of Texas at Austin, 
  (\email{krishnanunni@utexas.edu}.}}

\usepackage{amsopn}


%% file: references.bib
@book{Iserles2008FirstCourseNADE,
  author    = {Iserles, Arieh},
  title     = {A First Course in the Numerical Analysis of Differential Equations},
  edition   = {2},
  publisher = {Cambridge University Press},
  address   = {Cambridge},
  year      = {2008},
  isbn      = {9780521734905}
}

@book{Scott2011NumericalAnalysis,
  author    = {Scott, L. Ridgway},
  title     = {Numerical Analysis},
  publisher = {Princeton University Press},
  address   = {Princeton, NJ},
  year      = {2011},
  isbn      = {978-0691146867}
}

@book{folland1999real,
  author    = {Gerald B. Folland},
  title     = {Real Analysis: Modern Techniques and Their Applications},
  edition   = {2},
  publisher = {John Wiley \& Sons},
  year      = {1999},
  isbn      = {978-0471317166}
}

@book{hirsch1976differential,
  author    = {Morris W. Hirsch},
  title     = {Differential Topology},
  publisher = {Springer},
  year      = {1976},
  series    = {Graduate Texts in Mathematics},
  volume    = {33},
  doi       = {10.1007/978-1-4684-9449-5}
}

@book{StoerBulirsch1993,
  author    = {Stoer, Josef and Bulirsch, Roland},
  title     = {Introduction to Numerical Analysis},
  edition   = {2},
  publisher = {Springer},
  year      = {1993},
  series    = {Texts in Applied Mathematics},
  volume    = {12},
  doi       = {10.1007/978-1-4757-2101-1}
}

@book{Davis1975,
  author    = {Davis, Philip J.},
  title     = {Interpolation and Approximation},
  publisher = {Dover Publications},
  year      = {1975}
}

@book{GasperRahman2004,
  author    = {Gasper, George and Rahman, Mizan},
  title     = {Basic Hypergeometric Series},
  edition   = {2},
  publisher = {Cambridge University Press},
  year      = {2004},
  series    = {Encyclopedia of Mathematics and its Applications},
  volume    = {96}
}

@techreport{Lourakis2007,
  author      = {Lourakis, Manolis I. A.},
  title       = {Interpolating Using {H}ermite Splines},
  institution = {Institute of Computer Science, FORTH},
  year        = {2007},
  url         = {https://www.ics.forth.gr}
}

@article{Schumaker1973,
  author  = {Schumaker, Larry L.},
  title   = {On {H}ermite interpolation by splines},
  journal = {Journal of Approximation Theory},
  volume  = {9},
  number  = {1},
  pages   = {2--29},
  year    = {1973},
  doi     = {10.1016/0021-9045(73)90003-0}
}

@incollection{BirkhoffDeBoor1964,
  author    = {Birkhoff, Garrett and de Boor, Carl},
  title     = {Piecewise Polynomial Interpolation and Approximation},
  booktitle = {Approximation of Functions},
  editor    = {Hull, T. E.},
  pages     = {164--190},
  publisher = {Elsevier},
  year      = {1964}
}

@book{quarteroni2006numerical,
  title={Numerical mathematics},
  author={Quarteroni, Alfio and Sacco, Riccardo and Saleri, Fausto},
  volume={37},
  year={2006},
  publisher={Springer Science \& Business Media}
}

@book{apostol1967one,
  title={One-variable calculus, with an introduction to linear algebra},
  author={Apostol, Tom M},
  volume={1},
  year={1967},
  publisher={Blaisdell Waltham, Toronto, London}
}

@misc{abromowitz1972handbook,
  title={Handbook of mathematical functions},
  author={Abromowitz, Milton and Stegun, Irene A},
  year={1972},
  publisher={Dover, New York}
}

@article{cruz2003elementary,
  title={An elementary proof of error estimates for the trapezoidal rule},
  author={Cruz-Uribe, David and Neugebauer, CJ},
  journal={Mathematics magazine},
  volume={76},
  number={4},
  pages={303--306},
  year={2003},
  publisher={Taylor \& Francis}
}

@book{alma991022147289706011,
title = {Interpolation and approximation},
author = {Davis, Philip J.},
publisher = {Dover Publications},
pages = {37},
year = {1975},
}

@book{atkinson,
title = {An introduction to numerical analysis},
author = {Atkinson, Kendall E.},
publisher = {Wiley},
year = {1978},
pages = {159 -- 161}
}

@article{LampretVito2004AItH,
author = {Lampret, Vito},
address = {PHILADELPHIA},
copyright = {Copyright 2004 Society for Industrial and Applied Mathematics},
issn = {0036-1445},
journal = {SIAM review},
language = {eng},
number = {2},
pages = {311-328},
publisher = {Society for Industrial and Applied Mathematics},
title = {An Invitation to Hermite's Integration and Summation: A Comparison between Hermite's and Simpson's Rules},
volume = {46},
year = {2004},
}

@article{ELSINGERJASONR.2007AEPO,
author = {ELSINGER, JASON R.},
issn = {0031-952X},
journal = {Pi Mu Epsilon journal},
language = {eng},
number = {6},
pages = {353-357},
publisher = {National Honorary Mathematics Society},
title = {AN ELEMENTARY PROOF OF THE SIMPSON'S RULE ERROR FORMULA},
volume = {12},
year = {2007},
}

@misc{HaiD.D.2008AEPo,
author = {Hai, D. D. and Smith, R. C.},
address = {Washington},
copyright = {Copyright 2008 Mathematical Association of America (Incorporated)},
issn = {0025-570X},
journal = {Mathematics magazine},
language = {eng},
number = {4},
pages = {295-300},
publisher = {Mathematical Association of America},
title = {An Elementary Proof of the Error Estimates in Simpson's Rule},
volume = {81},
year = {2008},
}
